\documentclass[11pt,reqno]{amsart}
\usepackage{amsmath,amssymb,amsthm}
\usepackage[a4paper,left=3cm,right=3cm,top=3cm,bottom=3cm]{geometry}
\usepackage{hyperref}

\theoremstyle{plain}
\newtheorem{thm}{Theorem}[section]
\newtheorem{lem}[thm]{Lemma}
\newtheorem{prop}[thm]{Proposition}
\newtheorem{coro}[thm]{Corollary}
\theoremstyle{definition}
\newtheorem{ex}[thm]{Example}
\newtheorem{rem}[thm]{Remark}
\newcommand{\Aut}{\operatorname{Aut}}
\newcommand{\GL}{\operatorname{GL}}
\newcommand{\Span}{\operatorname{Span}}
\newcommand{\rank}{\operatorname{rank}}
\newcommand{\im}{\operatorname{Im}}

\title[Isotropy Groups of Locally Finite Derivations]{On the Isotropy Groups of Locally Finite Derivations of Laurent Polynomial Algebras}
\author{Anderson Porto}
\address{Universidade Federal dos Vales do Jequitinhonha e Mucuri, Brasil}%
\email{ander.porto@ict.ufvjm.edu.br}

\author{Marcelo Veloso} 
\address{Universidade Federal de S\~ao Jo\~ao del-Rei, Brasil}
\email{veloso@ufsj.edu.br}



\begin{document}
	\begin{abstract}
		Let \(A_n=\Bbbk[x_1^{\pm1},\ldots,x_n^{\pm1}]\) be a Laurent
		polynomial algebra over a field \(\Bbbk\) of characteristic zero.
		We determine explicitly the automorphisms commuting with each
		locally finite derivation
		\[
		D_\mu=\sum_{i=1}^n\mu_i x_i\frac{\partial}{\partial x_i},
		\qquad \mu=(\mu_1,\ldots,\mu_n)^T\in\Bbbk^n.
		\]
		These are precisely the automorphisms
		\(\rho(x_i)=\lambda_i x^{m_i}\), where
		\(\lambda_i\in\Bbbk^\ast\), the exponent vectors \(m_i\) are
		the columns of a matrix \(M_\rho\in\GL_n(\mathbb Z)\), and
		\(\mu^TM_\rho=\mu^T\).
		We also give an explicit parametrization in dimension two and relate
		these automorphisms to symmetries of diagonal flows on algebraic tori.
	\end{abstract}
	\maketitle
	
	\section*{Introduction}

	Locally finite derivations form an important class of derivations
	in commutative and noncommutative algebra. Recall that a derivation
	\(D\) of a \(\Bbbk\)-algebra \(A\) is called locally finite if, for
	every \(a\in A\), the \(\Bbbk\)-vector space
	\[
	\operatorname{Span}_{\Bbbk}\{D^m(a)\mid m\geq0\}
	\]
	is finite-dimensional. These derivations arise naturally in
	connection with polynomial flows, automorphism groups, and
	algebraic group actions.
	
	A fundamental result in the two-dimensional commutative setting
	is due to van den Essen \cite{V92}, who classified, up to
	conjugation, the locally finite derivations of the polynomial
	algebra \(\Bbbk[X,Y]\) over a field of characteristic zero.
	His classification provides a finite list of normal forms and
	has served as a starting point for investigations of the structure
	and symmetries of locally finite derivations.
	
	Given a derivation \(D\) of a \(\Bbbk\)-algebra \(A\), a natural
	problem is to determine its isotropy group
	\[
	\operatorname{Aut}_D(A)
	=
	\{\rho\in\operatorname{Aut}_{\Bbbk}(A)\mid \rho D=D\rho\}.
	\]
	This group consists of the automorphisms of \(A\) that preserve
	\(D\). In \cite{BFV}, the notion of the tame isotropy group of
	a derivation was introduced and studied for several classes of
	derivations of polynomial algebras, including triangular
	derivations, simple derivations in two variables, and simple
	Shamsuddin derivations. More recently, van den Essen's
	classification was used to obtain explicit descriptions of
	the isotropy groups of locally finite derivations of
	\(\Bbbk[X,Y]\) \cite{CidVeloso}.
	
	Locally finite derivations have also been investigated beyond
	the polynomial setting. In \cite{BBVS-LFD-Ore}, they were
	classified for Ore extensions of \(\Bbbk[x]\) over an
	algebraically closed field of characteristic zero, including
	the quantum plane, the first quantum Weyl algebra, and
	differential Ore extensions. This provides a noncommutative
	counterpart to van den Essen's classification in two variables.
	
	The purpose of the present paper is to determine explicitly
	the isotropy groups of locally finite derivations of the
	Laurent polynomial algebra
	\[
	A_n=\Bbbk[x_1^{\pm1},\ldots,x_n^{\pm1}],
	\]
	where \(\Bbbk\) is a field of characteristic zero and \(n\geq1\).
	This algebra is the coordinate ring of the split algebraic
	torus \(\mathbb G_m^n\). Its automorphisms have a monomial
	form, and its locally finite derivations admit a particularly
	simple description. Indeed, by Zhao
	\cite[Proposition~2.1]{Zhao}, every such derivation has the form
	\[
	D_\mu=\sum_{i=1}^n\mu_i x_i\frac{\partial}{\partial x_i},
	\qquad
	\mu=(\mu_1,\ldots,\mu_n)^T\in\Bbbk^n.
	\]
	Throughout the paper, weight vectors and exponent vectors
	are regarded as column vectors. In
	Proposition~\ref{isotropy-condition-An}, we show that an
	automorphism
	\[
	\rho(x_i)=\lambda_i x^{m_i},
	\qquad
	\lambda_i\in\Bbbk^\ast,
	\qquad
	M_\rho=(m_1\ \cdots\ m_n)\in\GL_n(\mathbb Z),
	\]
	commutes with \(D_\mu\) if and only if
	\[
	\mu^TM_\rho=\mu^T.
	\]
	This criterion also follows from the action of monomial
	automorphisms on Euler derivations described in
	\cite[Lemma~2.9]{Bavula16}. It reduces the isotropy problem
	to the determination of the unimodular integer matrices
	satisfying this condition.
	
	Our emphasis is on the explicit description of the commuting
	automorphisms. In dimension two, we give a parametrization
	of their exponents. In arbitrary dimension, we use the lattice
	of integer relations among the weights,
	\[
	L_\mu=\{\ell\in\mathbb Z^n:\mu^T\ell=0\},
	\]
	to write every commuting automorphism as
	\[
	\rho(x_i)=\lambda_i x^{e_i+\ell_i},
	\qquad
	\lambda_i\in\Bbbk^\ast,
	\qquad
	\ell_i\in L_\mu,
	\]
	where \(e_1,\ldots,e_n\) are the standard basis vectors of
	\(\mathbb Z^n\), subject to the condition
	\[
	(e_1+\ell_1\ \cdots\ e_n+\ell_n)\in\GL_n(\mathbb Z).
	\]
	Writing
	\[
	r=\dim_{\mathbb Q}
	\Span_{\mathbb Q}\{\mu_1,\ldots,\mu_n\},
	\]
	we choose an integral basis adapted to \(L_\mu\), in which
	the admissible exponent matrices are precisely
	\[
	\begin{pmatrix}
		I_r&0\\
		B&C
	\end{pmatrix},
	\qquad
	B\in M_{(n-r)\times r}(\mathbb Z),
	\quad
	C\in\GL_{n-r}(\mathbb Z).
	\]
	This yields a parametrization of all commuting automorphisms.
	The resulting abstract group structure is recorded in a
	remark following the classification.
	
	Finally, over \(\mathbb R\) or \(\mathbb C\), we interpret
	these automorphisms as algebraic symmetries of the diagonal
	flow associated with \(D_\mu\). The same lattice \(L_\mu\)
	also describes the Laurent monomial first integrals and
	the ring of constants of the derivation.
	\section{Automorphisms and Locally Finite Derivations of Laurent Polynomial Algebras}
	
	Throughout the paper, \(\Bbbk\) is a field of characteristic zero,
	\(n\geq1\), and
	\[
	A_n=\Bbbk[x_1^{\pm1},\ldots,x_n^{\pm1}].
	\]
	For a column vector \(\alpha=(\alpha_1,\ldots,\alpha_n)^T\in\mathbb Z^n\),
	we write \(x^\alpha=x_1^{\alpha_1}\cdots x_n^{\alpha_n}\).
	We denote by \(e_1,\ldots,e_n\) the standard basis vectors of
	\(\mathbb Z^n\), regarded as column vectors.
	
	\subsection{Automorphisms of Laurent polynomial algebras}
	
	The units of \(A_n\) are precisely the scalar multiples of Laurent
	monomials:
	\[
	A_n^\ast=\{\lambda x^\alpha:\lambda\in\Bbbk^\ast,
	\ \alpha\in\mathbb Z^n\}.
	\]
	Consequently, every \(\Bbbk\)-algebra automorphism \(\rho\) of \(A_n\)
	is uniquely determined by
	\[
	\rho(x_i)=\lambda_i x^{m_i},\qquad
	\lambda_i\in\Bbbk^\ast,\qquad 1\leq i\leq n,
	\]
	where
	\[
	m_i=(a_{1i},\ldots,a_{ni})^T,
	\qquad M_\rho=(m_1\ \cdots\ m_n)\in\GL_n(\mathbb Z).
	\]
	Indeed, \(\rho\) induces an automorphism of
	\(A_n^\ast/\Bbbk^\ast\simeq\mathbb Z^n\), represented by \(M_\rho\).
	Conversely, any such matrix and nonzero scalars define an automorphism
	of \(A_n\); see also \cite{Bavula16}.
	
	In particular, the \(i\)-th column of \(M_\rho\) is the exponent
	vector of the monomial in \(\rho(x_i)\). For every \(\alpha\in\mathbb Z^n\),
	\[
	\rho(x^\alpha)=\lambda^\alpha x^{M_\rho\alpha},
	\qquad \lambda^\alpha=\prod_{i=1}^n\lambda_i^{\alpha_i}.
	\]
	
	\begin{ex}
		Every automorphism of \(A_2\) has the form
		\[
		\rho(x_1)=\lambda_1x_1^ax_2^c,\qquad
		\rho(x_2)=\lambda_2x_1^bx_2^d,
		\]
		where \(\lambda_1,\lambda_2\in\Bbbk^\ast\) and
		\[
		M_\rho=\begin{pmatrix}a&b\\c&d\end{pmatrix},
		\qquad ad-bc=\pm1.
		\]
	\end{ex}
	
	\subsection{Locally finite derivations}
	
	For \(1\leq i\leq n\), set \(E_i=x_i\partial/\partial x_i\).
	Since
	\[
	E_i(x^\alpha)=\alpha_i x^\alpha,
	\]
	each \(E_i\) acts diagonally on the Laurent monomial basis.
	The following characterization is due to Zhao
	\cite[Proposition~2.1]{Zhao}.
	
	\begin{thm}\label{th:LFD-Laurent}
		A derivation \(D\in\operatorname{Der}_{\Bbbk}(A_n)\) is locally finite
		if and only if
		\[
		D=D_\mu=\sum_{i=1}^n\mu_i E_i
		\]
		for some \(\mu=(\mu_1,\ldots,\mu_n)^T\in\Bbbk^n\).
	\end{thm}
	
	For a derivation of this form,
	\[
	D_\mu(x^\alpha)=(\mu^T\alpha)x^\alpha.
	\]
	In particular, every Laurent polynomial belongs to the finite-dimensional
	\(D_\mu\)-invariant space spanned by the monomials in its support.
	We seek an explicit description of
	\[
	\Aut_{D_\mu}(A_n)=
	\{\rho\in\Aut_{\Bbbk}(A_n):\rho D_\mu=D_\mu\rho\}.
	\]
	When \(\mu=0\), every automorphism commutes with \(D_\mu\).
	Thus the nonzero derivations are the only cases requiring further analysis.
	
	\section{The Two-Dimensional Case}
	
	In this section, let
	\[
	D_\mu=\mu_1x_1\frac{\partial}{\partial x_1}
	+\mu_2x_2\frac{\partial}{\partial x_2},
	\qquad \mu=(\mu_1,\mu_2)^T\ne0.
	\]
	We first determine the conditions imposed on the exponents of an
	automorphism commuting with \(D_\mu\).
	
	\begin{prop}\label{prop:isotropy-condition-A2}
		An automorphism
		\[
		\rho(x_1)=\lambda_1x_1^ax_2^c,\qquad
		\rho(x_2)=\lambda_2x_1^bx_2^d
		\]
		commutes with \(D_\mu\) if and only if
		\[
		(a-1)\mu_1+c\mu_2=0,\qquad
		b\mu_1+(d-1)\mu_2=0.
		\]
		Equivalently, \(\mu^TM_\rho=\mu^T\).
	\end{prop}
	
	\begin{proof}
		For the two generators, we have
		\[
		\begin{aligned}
			\rho D_\mu(x_1)&=\mu_1\lambda_1x_1^ax_2^c,&
			D_\mu\rho(x_1)&=(a\mu_1+c\mu_2)\lambda_1x_1^ax_2^c,\\
			\rho D_\mu(x_2)&=\mu_2\lambda_2x_1^bx_2^d,&
			D_\mu\rho(x_2)&=(b\mu_1+d\mu_2)\lambda_2x_1^bx_2^d.
		\end{aligned}
		\]
		Equality on \(x_1,x_2\) gives exactly the stated conditions.
		It also implies equality on their inverses, by
		\(D(u^{-1})=-u^{-2}D(u)\), and then on all of \(A_2\) by the
		Leibniz rule and linearity. Finally,
		\[
		\mu^TM_\rho=(a\mu_1+c\mu_2,\ b\mu_1+d\mu_2),
		\]
		which gives the matrix formulation.
	\end{proof}
	
	\subsection{The \(\mathbb Q\)-linearly independent case}
	
	\begin{prop}\label{prop:independent-A2}
		If \(\mu_1,\mu_2\) are linearly independent over \(\mathbb Q\), then
		\(\Aut_{D_\mu}(A_2)\) consists precisely of the diagonal automorphisms
		\[
		\rho(x_1)=\lambda_1x_1,\qquad
		\rho(x_2)=\lambda_2x_2,
		\qquad \lambda_1,\lambda_2\in\Bbbk^\ast.
		\]
	\end{prop}
	
	\begin{proof}
		The two relations in Proposition~\ref{prop:isotropy-condition-A2}
		force \(a=d=1\) and \(b=c=0\). Conversely, every diagonal
		automorphism commutes with \(D_\mu\).
	\end{proof}
	
	\subsection{The \(\mathbb Q\)-linearly dependent case}
	
	Suppose that \(\mu_1,\mu_2\) are linearly dependent over \(\mathbb Q\).
	As \(\mu\ne0\), we can write
	\[
	\mu=\alpha(p,q)^T,
	\qquad \alpha\in\Bbbk^\ast,\quad p,q\in\mathbb Z,
	\quad \gcd(p,q)=1.
	\]
	Then \(\mu^TM=\mu^T\) is equivalent to \((p,q)M=(p,q)\).
	
	\begin{prop}\label{prop:dependent-A2}
		With the notation above, the automorphisms commuting with \(D_\mu\)
		are precisely
		\[
		\begin{aligned}
			\rho(x_1)&=\lambda_1x_1^{1+qu}x_2^{-pu},\\
			\rho(x_2)&=\lambda_2x_1^{qv}x_2^{1-pv},
		\end{aligned}
		\qquad
		\begin{gathered}
			\lambda_1,\lambda_2\in\Bbbk^\ast,\\
			u,v\in\mathbb Z,\qquad qu-pv\in\{0,-2\}.
		\end{gathered}
		\]
	\end{prop}
	
	\begin{proof}
		Write \(M_\rho=\begin{pmatrix}a&b\\c&d\end{pmatrix}\).
		By Proposition~\ref{prop:isotropy-condition-A2}, commutation is equivalent to
		\[
		p(a-1)+qc=0,\qquad pb+q(d-1)=0.
		\]
		Thus both columns of \(M_\rho-I_2\) belong to the kernel of
		\[
		\mathbb Z^2\longrightarrow\mathbb Z,
		\qquad (h,k)^T\longmapsto ph+qk.
		\]
		As \(\gcd(p,q)=1\), this kernel is generated by \((q,-p)^T\).
		Consequently, for some \(u,v\in\mathbb Z\),
		\[
		M_\rho-I_2=\begin{pmatrix}q\\-p\end{pmatrix}(u\ v),
		\qquad
		M_\rho=\begin{pmatrix}1+qu&qv\\-pu&1-pv\end{pmatrix}.
		\]
		Its determinant is \(1+qu-pv\). Therefore, this matrix is unimodular
		if and only if \(qu-pv\in\{0,-2\}\).
		Conversely, every matrix of this form with determinant \(\pm1\)
		fixes \((p,q)\) on the right and defines, for arbitrary nonzero
		\(\lambda_1,\lambda_2\), a commuting automorphism.
	\end{proof}
	
	\begin{rem}\label{rem:det-one-A2}
		The automorphisms in Proposition~\ref{prop:dependent-A2} whose
		exponent matrices have determinant one are obtained by taking
		\(u=pt\) and \(v=qt\), with \(t\in\mathbb Z\). Their exponent matrices are
		\[
		M_t=\begin{pmatrix}
			1+pqt&q^2t\\-p^2t&1-pqt
		\end{pmatrix}.
		\]
		This follows from \(qu-pv=0\) and \(\gcd(p,q)=1\).
	\end{rem}
	
	\begin{thm}\label{th:isotropy-LFD-A2}
		Let \(D_\mu\ne0\) be a locally finite derivation of \(A_2\).
		If \(\mu_1,\mu_2\) are linearly independent over \(\mathbb Q\),
		its commuting automorphisms are exactly the diagonal scalings.
		Otherwise, they are precisely the automorphisms parametrized in
		Proposition~\ref{prop:dependent-A2}.
	\end{thm}
	
	\section{The General Case}
	
	Let
	\[
	D_\mu=\sum_{i=1}^n\mu_i x_i\frac{\partial}{\partial x_i},
	\qquad \mu=(\mu_1,\ldots,\mu_n)^T\in\Bbbk^n\setminus\{0\}.
	\]
	We first express commutation in terms of the columns of the exponent matrix.
	
	\begin{prop}\label{isotropy-condition-An}
		Let \(\rho\in\Aut_{\Bbbk}(A_n)\), with
		\(\rho(x_i)=\lambda_i x^{m_i}\) and
		\(M_\rho=(m_1\ \cdots\ m_n)\).
		Then
		\[
		\rho D_\mu=D_\mu\rho
		\quad\Longleftrightarrow\quad
		\mu^Tm_i=\mu_i\ \text{for every }i
		\quad\Longleftrightarrow\quad
		\mu^TM_\rho=\mu^T.
		\]
	\end{prop}
	
	\begin{proof}
		For each \(i\),
		\[
		\rho D_\mu(x_i)=\mu_i\lambda_i x^{m_i},
		\qquad D_\mu\rho(x_i)=(\mu^Tm_i)\lambda_i x^{m_i}.
		\]
		These expressions agree exactly when \(\mu^Tm_i=\mu_i\).
		As in the two-dimensional case, equality on the \(x_i\) implies
		equality on their inverses and hence on all of \(A_n\).
		The conditions on the columns are equivalent to the stated matrix equality.
	\end{proof}
	
	Set
	\[
	L_\mu=\{\ell\in\mathbb Z^n:\mu^T\ell=0\},
	\qquad r=\dim_{\mathbb Q}\Span_{\mathbb Q}\{\mu_1,\ldots,\mu_n\}.
	\]
	Since \(\mu_i=\mu^Te_i\), the proposition says that each exponent
	vector has the form \(m_i=e_i+\ell_i\), with \(\ell_i\in L_\mu\).
	To describe all unimodular matrices with these columns, we choose
	an integral basis adapted to \(L_\mu\).
	
	\begin{lem}\label{lem:adapted-basis}
		There exists a basis
		\(u_1,\ldots,u_r,v_1,\ldots,v_s\) of \(\mathbb Z^n\), where
		\(s=n-r\), such that \(v_1,\ldots,v_s\) is a basis of \(L_\mu\).
		In particular, \(\rank_{\mathbb Z}L_\mu=n-r\).
	\end{lem}
	
	\begin{proof}
		Consider
		\[
		\varphi_\mu:\mathbb Z^n\longrightarrow(\Bbbk,+),
		\qquad a\longmapsto\mu^Ta.
		\]
		Its kernel is \(L_\mu\), and its image is the finitely generated
		torsion-free abelian group \(\mathbb Z\mu_1+\cdots+\mathbb Z\mu_n\),
		which has rank \(r\). Choose a basis \(\gamma_1,\ldots,\gamma_r\)
		of this image and lifts \(u_i\in\mathbb Z^n\) with
		\(\varphi_\mu(u_i)=\gamma_i\). As a subgroup of \(\mathbb Z^n\),
		\(L_\mu\) is free abelian; choose a basis \(v_1,\ldots,v_s\).
		For every \(z\in\mathbb Z^n\), write
		\(\varphi_\mu(z)=\sum_i a_i\gamma_i\). Then
		\(z-\sum_i a_i u_i\in L_\mu\), so the \(u_i\) and \(v_j\)
		generate \(\mathbb Z^n\). If
		\(\sum_i a_i u_i+\sum_j b_jv_j=0\), applying \(\varphi_\mu\)
		first gives all \(a_i=0\), and then all \(b_j=0\).
		Thus these vectors form a basis, and \(r+s=n\).
	\end{proof}
	
	\begin{prop}\label{prop:structure-Gmu}
		Choose a basis as in Lemma~\ref{lem:adapted-basis}, and set
		\[
		U=(u_1\ \cdots\ u_r\ v_1\ \cdots\ v_s)\in\GL_n(\mathbb Z).
		\]
		Then an integer matrix \(M\) is unimodular and satisfies
		\(\mu^TM=\mu^T\) if and only if
		\[
		M=U\begin{pmatrix}I_r&0\\B&C\end{pmatrix}U^{-1},
		\qquad
		B\in M_{s\times r}(\mathbb Z),\quad C\in\GL_s(\mathbb Z).
		\]
	\end{prop}
	
	\begin{proof}
		Put \(\gamma_i=\varphi_\mu(u_i)\). These elements form a basis of
		\(\im\varphi_\mu\) and are therefore linearly independent over
		\(\mathbb Q\). Moreover,
		\[
		\mu^TU=(\gamma_1,\ldots,\gamma_r,0,\ldots,0).
		\]
		For an integer matrix \(M\), write
		\[
		N=U^{-1}MU=\begin{pmatrix}A&H\\B&C\end{pmatrix}.
		\]
		Multiplying \(\mu^TM=\mu^T\) on the right by \(U\) gives
		\[
		(\gamma_1,\ldots,\gamma_r,0,\ldots,0)N
		=(\gamma_1,\ldots,\gamma_r,0,\ldots,0).
		\]
		The independence of the \(\gamma_i\) implies that this equality
		holds exactly when \(A=I_r\) and \(H=0\).
		For such a matrix, \(\det M=\det N=\det C\). Hence \(M\) is
		unimodular exactly when \(C\in\GL_s(\mathbb Z)\), with no
		restriction on \(B\).
	\end{proof}
	
	When \(s=0\), the block matrix in this proposition is simply \(I_n\).
	We use the convention \(\GL_0(\mathbb Z)=\{1\}\).
	
	\begin{thm}\label{th:isotropy-LFD-An}
		Let \(D_\mu\ne0\) be a locally finite derivation of \(A_n\).
		The automorphisms commuting with \(D_\mu\) are precisely
		\[
		\rho(x_i)=\lambda_i x^{e_i+\ell_i},\qquad
		\lambda_i\in\Bbbk^\ast,\quad \ell_i\in L_\mu,\quad 1\leq i\leq n,
		\]
		for which
		\[
		(e_1+\ell_1\ \cdots\ e_n+\ell_n)\in\GL_n(\mathbb Z).
		\]
		Equivalently, fix the matrix \(U\) from
		Proposition~\ref{prop:structure-Gmu}. All such automorphisms, and only
		these, are obtained by choosing arbitrary
		\[
		\lambda_1,\ldots,\lambda_n\in\Bbbk^\ast,\qquad
		B\in M_{s\times r}(\mathbb Z),\qquad C\in\GL_s(\mathbb Z),
		\]
		setting
		\[
		M=U\begin{pmatrix}I_r&0\\B&C\end{pmatrix}U^{-1},
		\]
		and defining \(\rho(x_i)=\lambda_i x^{Me_i}\).
	\end{thm}
	
	\begin{proof}
		The first description follows from
		Proposition~\ref{isotropy-condition-An}, since
		\(\mu^Tm_i=\mu_i\) is equivalent to \(m_i-e_i\in L_\mu\).
		The parametrization of all admissible matrices is given by
		Proposition~\ref{prop:structure-Gmu}.
	\end{proof}
	
	\begin{coro}\label{cor:independent-An}
		If \(\mu_1,\ldots,\mu_n\) are linearly independent over \(\mathbb Q\),
		then the automorphisms commuting with \(D_\mu\) are precisely
		\[
		\rho(x_i)=\lambda_i x_i,\qquad
		\lambda_i\in\Bbbk^\ast,\quad 1\leq i\leq n.
		\]
	\end{coro}
	
	\begin{proof}
		In this case \(L_\mu=\{0\}\), so the assertion follows from
		Theorem~\ref{th:isotropy-LFD-An}.
	\end{proof}
	
	\begin{rem}[Abstract group structure]\label{rem:group-structure}
		The preceding descriptions also identify the abstract isotropy group.
		Set \(G_\mu=\{M\in\GL_n(\mathbb Z):\mu^TM=\mu^T\}\).
		In the adapted basis, multiplication is given by
		\[
		\begin{pmatrix}I_r&0\\B&C\end{pmatrix}
		\begin{pmatrix}I_r&0\\B'&C'\end{pmatrix}
		=\begin{pmatrix}I_r&0\\B+CB'&CC'\end{pmatrix}.
		\]
		Consequently,
		\[
		G_\mu\simeq\mathbb Z^{r(n-r)}\rtimes\GL_{n-r}(\mathbb Z),
		\]
		where \(C\) acts on \(B\) by left multiplication.
		The map \(\rho\mapsto M_\rho\) is a surjective homomorphism from
		\(\Aut_{D_\mu}(A_n)\) to \(G_\mu\), with kernel the diagonal scalings,
		and it splits by \(M\mapsto\sigma_M\), where
		\(\sigma_M(x^\alpha)=x^{M\alpha}\). Thus
		\[
		\Aut_{D_\mu}(A_n)\simeq G_\mu\ltimes(\Bbbk^\ast)^n,
		\]
		with the torus as the normal factor. Explicitly, if
		\(\tau_\lambda(x_i)=\lambda_i x_i\), then
		\[
		\sigma_M\tau_\lambda\sigma_M^{-1}=\tau_{M\cdot\lambda},
		\qquad (M\cdot\lambda)_i=\lambda^{M^{-1}e_i}.
		\]
		For \(n=2\) and \(r=1\), the matrix group is
		\(\mathbb Z\rtimes\{\pm1\}\), where \(-1\) acts by negation;
		this is the infinite dihedral group \(D_\infty\).
		For \(r=n\), the isotropy group is \((\Bbbk^\ast)^n\).
		If \(\mu=0\), every automorphism commutes with \(D_\mu\), and
		\(\Aut_{D_\mu}(A_n)\simeq\GL_n(\mathbb Z)\ltimes(\Bbbk^\ast)^n\).
	\end{rem}
	
	\section{Isotropy Groups and Flows on Algebraic Tori}
	
	In this section, assume that \(\Bbbk=\mathbb R\) or \(\mathbb C\),
	and let \(t\in\Bbbk\). The differential system associated with
	\(D_\mu\) is
	\[
	\frac{dx_i(t)}{dt}=\mu_i x_i(t),\qquad 1\leq i\leq n.
	\]
	Its solution defines the diagonal flow
	\[
	\Phi_t(z_1,\ldots,z_n)
	=(e^{\mu_1t}z_1,\ldots,e^{\mu_nt}z_n)
	\]
	on \((\Bbbk^\ast)^n\). This flow is analytic in \(t\); for each fixed
	\(t\), it is an algebraic automorphism of the torus and extends to a
	linear diagonal automorphism of \(\Bbbk^n\). In particular, it is
	polynomial in the initial conditions, in accordance with the notion
	of polynomial flow considered in \cite{V92}. The dependence on
	\(t\) need not be algebraic, so this is generally not an algebraic
	\(\mathbb G_a\)-action.
	
	We distinguish \(\Phi_t\) from its pullback on the coordinate ring,
	\[
	\Phi_t^\ast:A_n\longrightarrow A_n,\qquad
	\Phi_t^\ast(f)=f\circ\Phi_t.
	\]
	Thus \(\Phi_t^\ast(x_i)=e^{\mu_i t}x_i\).
	For an automorphism \(\rho(x_i)=\lambda_i x^{m_i}\), let
	\[
	F_\rho(z)=(\lambda_1z^{m_1},\ldots,\lambda_nz^{m_n})
	\]
	be the corresponding automorphism of the torus, so that
	\(F_\rho^\ast=\rho\).
	
	\begin{prop}\label{prop:flow-isotropy}
		For \(\rho\in\Aut_{\Bbbk}(A_n)\), the following are equivalent:
		\begin{enumerate}
			\item \(\rho D_\mu=D_\mu\rho\);
			\item \(\mu^TM_\rho=\mu^T\);
			\item \(\rho\Phi_t^\ast=\Phi_t^\ast\rho\) for every \(t\in\Bbbk\);
			\item \(F_\rho\circ\Phi_t=\Phi_t\circ F_\rho\) for every \(t\in\Bbbk\).
		\end{enumerate}
	\end{prop}
	
	\begin{proof}
		The equivalence between (1) and (2) follows from
		Proposition~\ref{isotropy-condition-An}.
		We prove that (2) and (3) are equivalent. For each \(i\), write
		\(m_i=(m_{1i},\ldots,m_{ni})^T\). Thus
		\[
		\rho(x_i)=\lambda_i x^{m_i}
		=\lambda_i x_1^{m_{1i}}\cdots x_n^{m_{ni}}.
		\]
		By definition, \(\Phi_t^\ast(f)=f\circ\Phi_t\), so
		$
		\Phi_t^\ast(x_j)=e^{\mu_jt}x_j$, for $j=1,\ldots,n$. Since \(\rho\in\Aut_{\Bbbk}(A_n)\), we have
		\[
		(\rho\circ\Phi_t^\ast)(x_i)
		=\rho(e^{\mu_i t}x_i)
		=e^{\mu_i t}\lambda_i x^{m_i}.
		\]
		On the other hand, \(\Phi_t^\ast\) is a
		\(\Bbbk\)-automorphism, and hence
		\[
		\begin{aligned}
			(\Phi_t^\ast\circ\rho)(x_i)
			&=\Phi_t^\ast\left(
			\lambda_i x_1^{m_{1i}}\cdots x_n^{m_{ni}}
			\right)\\
			&=\lambda_i
			\left(e^{\mu_1t}x_1\right)^{m_{1i}}
			\cdots
			\left(e^{\mu_nt}x_n\right)^{m_{ni}}\\
			&=\lambda_i
			e^{(\mu_1m_{1i}+\cdots+\mu_nm_{ni})t}
			x_1^{m_{1i}}\cdots x_n^{m_{ni}}\\
			&=\lambda_i e^{(\mu^Tm_i)t}x^{m_i}.
		\end{aligned}
		\]
		Therefore,
		\[
		\rho\circ\Phi_t^\ast=\Phi_t^\ast\circ\rho
		\quad\text{for every }t\in\Bbbk
		\]
		if and only if
		\[
		e^{\mu_i t}=e^{(\mu^Tm_i)t}
		\quad\text{for every }t\in\Bbbk
		\text{ and every }i.
		\]
		Differentiating these identities at \(t=0\) gives
		\(\mu_i=\mu^Tm_i\) for every \(i\). Conversely, these
		equalities imply the exponential identities above.
		As the \(m_i\) are the columns of \(M_\rho\), this is
		equivalent to
		\[
		\mu^TM_\rho=\mu^T.
		\]
		This proves the equivalence between (2) and (3).
		
		Finally, \(F_\rho^\ast=\rho\), and pullback reverses
		the order of composition. Thus
		\[
		(\Phi_t\circ F_\rho)^\ast
		=\rho\circ\Phi_t^\ast,
		\qquad
		(F_\rho\circ\Phi_t)^\ast
		=\Phi_t^\ast\circ\rho.
		\]
		Since morphisms of affine varieties are determined by
		their pullbacks on coordinate rings, (3) and (4)
		are equivalent.
	\end{proof}

	The explicit automorphisms described above therefore give all algebraic
	symmetries commuting with this flow. The correspondence
	\(\rho\mapsto F_\rho\) reverses composition; if a group isomorphism
	is desired, it is given by \(\rho\mapsto F_{\rho^{-1}}\).
	
	The lattice \(L_\mu\) also describes the first integrals of the system.
	
	\begin{prop}\label{prop:first-integrals}
		For \(\ell\in\mathbb Z^n\), the Laurent monomial \(x^\ell\) is
		constant along the flow if and only if \(\ell\in L_\mu\). Moreover,
		\[
		A_n^{D_\mu}:=\ker D_\mu
		=\bigoplus_{\ell\in L_\mu}\Bbbk x^\ell.
		\]
		If \(v_1,\ldots,v_s\) is a basis of \(L_\mu\), where \(s=n-r\), then
		\[
		A_n^{D_\mu}
		=\Bbbk[x^{v_1},x^{-v_1},\ldots,x^{v_s},x^{-v_s}]
		\simeq\Bbbk[t_1^{\pm1},\ldots,t_s^{\pm1}].
		\]
		For \(s=0\), this algebra is understood to be \(\Bbbk\).
	\end{prop}
	
	\begin{proof}
		We have
		\[
		\Phi_t^\ast(x^\ell)=e^{(\mu^T\ell)t}x^\ell,
		\qquad D_\mu(x^\ell)=(\mu^T\ell)x^\ell.
		\]
		Thus \(x^\ell\) is constant along every trajectory exactly when
		\(\mu^T\ell=0\). The description of the kernel follows because
		Laurent monomials form a basis of eigenvectors for \(D_\mu\).
		Every \(\ell\in L_\mu\) is an integral combination of the \(v_j\),
		which gives the displayed generators. Finally, the map
		\(t_j\mapsto x^{v_j}\) is injective because distinct integer
		combinations of the \(v_j\) give distinct Laurent monomials.
	\end{proof}
	
	Although stated here in connection with flows, the description of
	\(\ker D_\mu\) is valid over any field of characteristic zero.
	When \(r=n\), there are no nonzero integer relations among the weights:
	\[
	L_\mu=\{0\},\qquad A_n^{D_\mu}=\Bbbk.
	\]
	In this case, the only algebraic automorphisms commuting with the flow
	are the diagonal scalings. More generally, the lattice \(L_\mu\)
	determines both the Laurent monomial first integrals and the allowable
	changes in the exponent vectors of commuting automorphisms.
	
	\bibliographystyle{alpha}
	\bibliography{biblio}
\end{document}